\documentclass[11pt,reqno]{amsart}
\usepackage{amsmath, amssymb, amsthm}
\usepackage{url}
\usepackage{hyperref}
\usepackage{tikz}
\usepackage{ytableau}
\usepackage[utf8]{inputenc}
\usepackage{xcolor}
\usepackage{tikz-cd}
\usepackage{hyperref}
\usepackage{nccmath}

\usepackage{youngtab}

\usetikzlibrary{decorations.pathreplacing}

\usetikzlibrary{decorations.pathreplacing}

\renewcommand{\(}{\left\(}
\renewcommand{\)}{\right\)}
\renewcommand{\[}{\left\[}
\renewcommand{\]}{\right\]}

\numberwithin{equation}{section}
\theoremstyle{plain}
\newtheorem{theorem}{Theorem}[section]
\newtheorem{lemma}[theorem]{Lemma}

\newtheorem{remark}[theorem]{Remark}

\newtheorem{problem}[theorem]{Problem}

\newtheorem{corollary}[theorem]{Corollary}

\makeatletter
\def\proof{\@ifnextchar[{\@oproof}{\@nproof}}
\def\@oproof[#1][#2]{\trivlist\item[\hskip\labelsep\textit{#2 Proof of\
		#1.}~]\ignorespaces}
\def\@nproof{\trivlist\item[\hskip\labelsep\textit{Proof.}~]\ignorespaces}

\makeatother

\begin{document}
	
	\title{Inequalities for the overpartition function arising from determinants} 
	
	\author{Gargi Mukherjee}
	\address{Institute for Algebra, Johannes Kepler University, Altenberger Straße 69, A-4040 Linz, Austria.}
	\email{gargi.mukherjee@dk-compmath.jku.at}
	
	\maketitle
	
	\begin{abstract}
Let $\overline{p}(n)$ denote the overpartition funtion. This paper presents the $2$-$\log$-concavity property of $\overline{p}(n)$ by considering a more general inequality of the following form
\begin{equation*}
\begin{vmatrix}
\overline{p}(n) & \overline{p}(n+1) & \overline{p}(n+2) \\ 
\overline{p}(n-1) & \overline{p}(n) & \overline{p}(n+1) \\ 
\overline{p}(n-2) & \overline{p}(n-1) & \overline{p}(n)
\end{vmatrix} > 0,
\end{equation*}
which holds for all $n \geq 42$.
	\end{abstract}

\hspace{0.65 cm} \textbf{Mathematics Subject Classifications.} Primary 05A20; 11C20.\\
\vspace{0.3 cm}
\hspace{0.8 cm} \textbf{Keywords.} Overpartition; determinant; 2-log-concavity.
\section{Introduction}
A sequence $a_{0},a_{1},\dots,a_{n}$ of real numbers is said to be unimodal if for some $0\leq k\leq n$ we have $a_{0}\leq a_{1}\leq \dots \leq a_{k-1}\leq a_{k} \geq a_{k+1}\geq\dots\geq a_{n}$ and $\log$-concave if $a^2_{j}\geq a_{j-1}a_{j+1}$ for all $1\leq j\leq n-1$. Note that a $\log$-concave sequence of positive numbers is unimodal. We say that the sequence $a_{0},a_{1},\dots,a_{n}$ has no internal zeros if there do not exist integers $0\leq i<j<k\leq n$ satisfying $a_{i} \neq 0$, $a_{j} = 0$ and $a_{k} \neq 0$. Then a nonnegative $\log$-concave sequence with no internal zeroes is unimodal. The study of $\log$-concavity problems of sequences shares an intimate connection with zeros of polynomials, due to Newton through the following result; see for example \cite[p. 52]{HLP}.
\begin{theorem}\label{Newton}
	Let
	\begin{equation*}
	P(x)=\sum_{j=0}^{n}\binom{n}{j}a_{j}x^j
	\end{equation*}
be a (real) polynomial with real zeros. Then $a^2_{j}\geq a_{j-1}a_{j+1}$ for all $j$.
\end{theorem}
We call a real polynomial $P(x) = \sum_{j=0}^{n}a_{j}x^j$ $\log$-concave if its coefficient sequence is $\log$-concave. More generally, the following theorem provides a necessary condition for a real polynomial to have only real zeros.
\begin{theorem}$($\cite[Theorem 1.2.1]{Brenti}$)$\label{Brenti}
Let $P(x)=\sum_{i=0}^{d}a_{j}x^j$ be a polynomial with nonnegative coefficients with only real zeros. Then the sequence $\{a_k\}_{0\leq k \leq d}$ is $\log$-concave with no internal zeros; in particular, it is unimodal.
\end{theorem}
Before we state a more general version of Theorem \ref{Brenti}, let us introduce the theory of total positivity in brevity. A matrix $A$ with entries in real number is called totally positive if the determinant of each of its minors is nonnegative. For a sequence $(a_{k})_{k\geq 0}$, define its Toeplitz matrix $T$ by
\begin{equation}
T = \begin{pmatrix}
a_0 & & & & \\
a_1 & a_0 & & & \\
a_2 & a_1 & a_0 & & \\
a_3 & a_2 & a_1 & a_0 \\
\vdots & \vdots & \vdots & \vdots & \ddots
\end{pmatrix}.
\end{equation}
We say that $(a_{k})_{k\geq 0}$ is a totally positive sequence or P\'{o}lya frequency sequence if its associated Toeplitz matrix is totally positive. The following theorem shares a deep entanglement between combinatorics of zeros of polynomials and total positivity of its coefficient sequence.
\begin{theorem}[Aissen et al. \cite{AESW}]\label{Aissen}
	Let $P(x)=\sum_{i=0}^{d}a_{j}x^j$ be a real polynomial with nonnegative coefficients. Then $P(x)$ has only real zeros if and only if its coefficient sequence is a P\'{o}lya frequency sequence.
\end{theorem}  
Associated to a real polynomial $P(x)=\sum_{i=0}^{d}a_{j}x^j$ with only real zeros, a sequence $(\gamma_{k})_{k\geq 0}$ is called multiplier sequence if the corresponding polynomial $P_{\gamma}(x)=\sum_{i=0}^{d}\gamma_{j}a_{j}x^j$ also has only real zeros. This theory laid its foundation in the seminal work of P\'{o}lya and Schur \cite{PS}. In this context, Craven and Csordas obtained the following theorem that further reduces to $2$-$\log$-concavity $(\text{or double Tur\'{a}n inequality})$ of a multiplier sequence. 
\begin{theorem}[Theorem 2.13, \cite{CS}]\label{CS}
	If $(\gamma_{k})_{k\geq 0}$ with $\gamma_{k}>0$, is a multiplier sequence, then 
	\begin{equation}\label{CSeqn}
	\begin{vmatrix}
	\gamma_{k} & \gamma_{k+1} & \gamma_{k+2}\\
	\gamma_{k-1} & \gamma_{k} & \gamma_{k+1}\\
	\gamma_{k-2} & \gamma_{k-1} & \gamma_{k}\\
	\end{vmatrix} \geq 0, \ \ \text{for}\ k \in \mathbb{Z}_{\geq 2}.
	\end{equation}
\end{theorem}
Since  $\gamma_{k}>0$, \eqref{CSeqn} is equivalent to the following inequality
\begin{equation*}
(\gamma^2_{k}-\gamma_{k-1}\gamma_{k+1})^2-(\gamma^2_{k-1}-\gamma_{k-2}\gamma_{k})(\gamma^2_{k+1}-\gamma_{k}\gamma_{k+2})\geq 0,
\end{equation*}
which amounts to say that $(\gamma_{k})_{k\geq 0}$ is $2$-$\log$-concave.\\
Now we turn to discuss briefly about partitions through the lens of combinatorial analysis, as discussed in the previous paragraph. A partition of a positive integer $n$ is a nonincreasing sequence $(\lambda_{1},\lambda_{2},\dots,\lambda_{r})$ of positive integers with $\lambda_{1}+\lambda_{2}+\dots+\lambda_{r}=n$ and $p(n)$ denotes the number of partitions of $n$. The systematic study of partitions dates back to Euler. Rigorous analytic approach comes into play in the theory of partitions since the foundational work of Hardy and Ramanujan \cite{RamanujanHardy}. Hardy and Ramanujan employed the celebrated circle method in order to explicitly describe the asymptotics of $p(n)$, specifically, given by
\begin{equation}\label{HardyRamanujan}
p(n) \sim \frac{1}{4\sqrt{3}n}e^{\pi\sqrt{\frac{2n}{3}}} \ \ (n \rightarrow \infty).
\end{equation} 
Later Rademacher \cite{Rademacher} refined the formulation of Hardy and Ramanujan to set a convergent series expression for $p(n)$ and an error bound was given due to Lehmer \cite{Lehmer}. Log-concavity of $p(n)$ has been studied independently by Nicolas \cite{Nicolas} and by DeSalvo and Pak \cite{DeSalvoPak} by confirming a conjecture of Chen \cite{Chentalk}. Since then the study on inequalities of the partition function from combinatorial analysis perspective has been documented in the works of Chen et al. \cite{Chen2}, \cite{Chen}. Recently Griffin, Ono, Rolen and Zagier \cite{GORZ} consider a more general paradigm to trace the zeros of a certain polynomial, called Jensen polynomials associated with $p(n)$, defined by
\begin{equation*}
J^{d,n}_{p}(x) = \sum_{j=0}^{d}\binom{d}{j}p(n+j)x^j.
\end{equation*}
For a more detail study on hyperbolicity of $J^{d,n}_{p}(x)$, see \cite{LarsonWagner}.\\
Corteel and Lovejoy \cite{CorteelLovejoy} initiated a broad generalization of partitions, called overpartition that offers a panorama of combinatorial perspective of basic hypergeometric series. An overpartition of $n$ is a nonincreasing sequence of natural numbers whose sum is $n$ in which the first occurrence (equivalently, the final occurrence) of a number
may be overlined and $\overline{p}(n)$ denotes the number of overpartitions of $n$. For convenience, define $\overline{p}(0)=1$. For example, there are $8$ overpartitions of $3$ enumerated by $3, \overline{3}, 2+1, \overline{2}+1, 2+\overline{1}, \overline{2}+\overline{1},1+1+1, \overline{1}+1+1$. Similar to the Hardy-Ramanujan-Rademacher type formula for $\overline{p}(n)$, Zuckerman \cite{Zuckerman} showed that
\begin{equation}\label{Zuckerman}
\overline{p}(n)=\frac{1}{2\pi}\underset{2 \nmid k}{\sum_{k=1}^{\infty}}\sqrt{k}\underset{(h,k)=1}{\sum_{h=0}^{k-1}}\dfrac{\omega(h,k)^2}{\omega(2h,k)}e^{-\frac{2\pi i n h}{k}}\dfrac{d}{dn}\biggl(\dfrac{\sinh \frac{\pi \sqrt{n}}{k}}{\sqrt{n}}\biggr),
\end{equation}
where
\begin{equation*}
\omega(h,k)=\text{exp}\Biggl(\pi i \sum_{r=1}^{k-1}\dfrac{r}{k}\biggl(\dfrac{hr}{k}-\biggl\lfloor\dfrac{hr}{k}\biggr\rfloor-\dfrac{1}{2}\biggr)\Biggr)
\end{equation*}
for positive integers $h$ and $k$. In order to prove $\log$-concavity of $\overline{p}(n)$, Engel \cite{Engel} provided an error term for $\overline{p}(n)$
\begin{equation}\label{Engel1}
\overline{p}(n)=\frac{1}{2\pi}\underset{2 \nmid k}{\sum_{k=1}^{N}}\sqrt{k}\underset{(h,k)=1}{\sum_{h=0}^{k-1}}\dfrac{\omega(h,k)^2}{\omega(2h,k)}e^{-\frac{2\pi i n h}{k}}\dfrac{d}{dn}\biggl(\dfrac{\sinh \frac{\pi \sqrt{n}}{k}}{\sqrt{n}}\biggr)+R_{2}(n,N),
\end{equation}
where
\begin{equation}\label{Engel2}
\bigl|R_{2}(n,N)\bigr|< \dfrac{N^{5/2}}{\pi n^{3/2}} \sinh \biggl(\dfrac{\pi \sqrt{n}}{N}\biggr).
\end{equation}
Along the lines of works of Chen et al. in context of the partition function, somewhat similar research works on inequalities for $\overline{p}(n)$ has already been recorded in \cite{WangXieZhang} and \cite{LiuZhang}.\\
Jia and Wang \cite{JiaWang} examined determinantal inequalities for $p(n)$ arising from the theory of total positivity, by set up the following theorem. As a corollary of Theorem \ref{JiaWang}, they proved $2$ $\log$-concavity of $p(n)$.
\begin{theorem}[Theorem 1.5, \cite{JiaWang}]\label{JiaWang}
	Let $p(n)$ denote the partition function and 
	\begin{equation}\label{JiaWangeqn1}
	M_{3}(p(n)) = \begin{pmatrix}
	p(n) & p(n+1) & p(n+2) \\ 
	p(n-1) & p(n) & p(n+1) \\ 
	p(n-2) & p(n-1) & p(n)
	\end{pmatrix}
	\end{equation}
Then for all $n \geq 222$, we have 
\begin{eqnarray}\label{JiaWangeqn2}
\det M_{3}(p(n)) &> &0.
\end{eqnarray}
\end{theorem}
In this paper, our primary goal is to prove $2$-$\log$-concavity
of $\overline{p}(n)$. In order to prove this, we set up a similar device as that of Theorem \ref{JiaWang} but in context of overpartitions. In particular, we shall prove the following result.
\begin{theorem}\label{Mainthm}
Let $\overline{p}(n)$ denote the overpartition funtion. Then for all $n \geq 42$, we have
\begin{equation}\label{Mainthmeqn}
\begin{vmatrix}
\overline{p}(n) & \overline{p}(n+1) & \overline{p}(n+2) \\ 
\overline{p}(n-1) & \overline{p}(n) & \overline{p}(n+1) \\ 
\overline{p}(n-2) & \overline{p}(n-1) & \overline{p}(n)
\end{vmatrix} > 0.
\end{equation}
\end{theorem}
Theorem \ref{Mainthm} straight away implies the $2$-$\log$-concavity of $\overline{p}(n)$, precisely
\begin{theorem}\label{Mainthmcor}
For all $n \geq 42$,
\begin{equation}\label{Mainthmcoreqn}
(\overline{p}(n)^2-\overline{p}(n-1)\overline{p}(n+1))^2-(\overline{p}(n-1)^2-\overline{p}(n-2)\overline{p}(n))\ (\overline{p}(n+1)^2-\overline{p}(n)\overline{p}(n+2))> 0.
\end{equation}
\end{theorem}
We organize this paper in the following format. In Section \ref{setup}, we set up the premise first by introducing an inequality for $\overline{p}(n)$ and reformulate Theorem \ref{Mainthm}, given in Theorem \ref{Mainthmreform}, followed by the documentation of Theorem \ref{thm2} and \ref{thm3}. This foundation enables us to provide the proof of Theorem \ref{Mainthmreform} in Section \ref{mainthm}. In the end, we discuss how one can guess an infinite number of inequalities for the overpartition function by considering totally positive matrix of order $k\times k$ with $k \in \mathbb{Z}_{\geq 2}$, described in Problem \ref{prob}. 
\section{Inequality for $\overline{p}(n)$ and its consequences}\label{setup}
The principal aim of this section is to construct the machinery in order to prove Theorem \ref{Mainthm}, the primary objective of this paper. To some extent, we follow a similar line of argument as described in the work of Jia and Wang \cite{JiaWang}. We will see that the Theorem \ref{thm2} and \ref{thm3} are the key tools to prove Theorem \ref{Mainthmreform}, a reprise version of Theorem \ref{Mainthm}. First we need to estimate the quotient $u_n=\frac{\overline{p}(n-1)\overline{p}(n+1)}{\overline{p}(n)^2}$ by showing its upper and lower bound $g(n)$ and $f(n)$ respectively (cf. Theorem \ref{thm1}), derived from the inequality $B_1(n)<\overline{p}(n)<B_2(n)$ with $B_1(n), B_2(n)$ given in \eqref{eqn1} and as an immediate consequence, we get the inequality \eqref{thm1coreqn} for $s(n)=u_{n-1}+u_{n+1}-u_{n-1}u_{n+1}$ as follows; for all $n \geq 3$,
\begin{equation*}
s_1(n)<s(n)<s_2(n),
\end{equation*}
 where $s_1(n)$ and $s_2(n)$ are combinations of $f(n+1), f(n-1), g(n+1)$, and $g(n-1)$. Therefore, to prove Theorem \ref{thm2} and \ref{thm3}, the principal idea behind it is to approximate $s_2(n)$ (cf. \eqref{eqn8}) and $s_1(n)g(n)^2-2g(n)+1$ (cf. \eqref{thm3eqn2}) by rational functions in $y\ (=y(n)=\pi\sqrt{n})$ (cf. \eqref{thm2eqn22} and \eqref{thm3eqn10}). In order to arrive at such estimation to ease the computation, it is necessary to bound the error term $\widetilde{T}(n)$ by $\mu(n)^{-2m}$ for some $m \in \mathbb{Z}_{\geq 1}$ because our estimation turns out to get a suitable polynomial approximation of $r,\ x,\ z$ and $w$ (cf. \eqref{thm2eqn1}) in terms of $y$.  In our case, it is sufficient to consider $m=3$, as stated in Lemma \ref{lemma1}.\\
We denote $\mu(n)= \pi\sqrt{n}$ and define 
\begin{eqnarray}\label{eqn1}
B_{1}(n) &=&\frac{e^{\mu(n)}}{8n}\Bigl(1-\frac{1}{\mu(n)}-\frac{1}{\mu(n)^6}\Bigr)\nonumber\\
\text{and}\ \ B_{2}(n) &=&\frac{e^{\mu(n)}}{8n}\Bigl(1-\frac{1}{\mu(n)}+\frac{1}{\mu(n)^6}\Bigr).
\end{eqnarray}
\begin{lemma}\label{lemma1}
	For all $n \geq 94$, we have
\begin{equation}\label{lemma1eqn}
B_{1}(n) < \overline{p}(n) < B_{2}(n).
\end{equation}
\end{lemma}
\begin{proof}
From \cite[eqn. (3.5)]{LiuZhang}, it follows that
\begin{equation}\label{lemma1eqn1}
\overline{p}(n) = \frac{e^{\mu(n)}}{8n} \Bigl(1-\frac{1}{\mu(n)}+\widetilde{T}(n)\Bigr)
\end{equation}
where 
\begin{equation*}
\widetilde{T}(n)= \Bigl(1+\dfrac{1}{\mu(n)}\Bigr)e^{-2\mu(n)}+\dfrac{8n}{e^{\mu(n)}}R_2(n,2),
\end{equation*}
 and $R_2(n,2)$ is the error term of \eqref{Zuckerman}, given in \eqref{Engel2}. By \cite[eqn. (3.6)]{LiuZhang}, we have
\begin{equation}\label{lemma1eqn2}
|\widetilde{T}(n)|<10\ e^{-\frac{1}{2}\mu(n)}.
\end{equation}
Now,
\begin{equation}\label{lemma1eqn3}
10\ e^{-\frac{1}{2}\mu(n)} < \frac{1}{\mu(n)^6}\ \ \  \text{for all}\ n \geq 275.
\end{equation}
\eqref{lemma1eqn1}-\eqref{lemma1eqn3} altogether imply \eqref{lemma1eqn} for all $n \geq 275$. We finish the proof by confirming \eqref{lemma1eqn} by checking numerically in Mathematica for all $94 \leq n \leq 274$.
\end{proof}
\begin{remark}\label{remark}
We note that the upper bound of the absolute value of error term $\widetilde{T}(n)$ can be improved by considering a more generalized version of \eqref{lemma1eqn3} of the following form: there exists $N(m) \in \mathbb{Z}_{\geq 1}$, such that for all $n \geq N(m)$,
\begin{equation}\label{remarkeqn}
10\ e^{-\frac{1}{2}\mu(n)} < \frac{1}{\mu(n)^m}.  
\end{equation}
\end{remark}
Let
\begin{equation}\label{eqn2}
u_{n}= \frac{\overline{p}(n-1)\overline{p}(n+1)}{\overline{p}(n)^2}
\end{equation}
and consequently, denote 
\begin{equation}\label{eqn3}
s(n) = u_{n-1}+u_{n+1}-u_{n-1}u_{n+1}.
\end{equation} 
Following the notations as given in \cite{JiaWang}, we set 
\begin{equation}\label{eqn4}
r =\mu(n-2),\ \ x=\mu(n-1),\ \ y=\mu(n),\ \ z=\mu(n+1),\ \ w=\mu(n+2),
\end{equation}
and
\begin{eqnarray}\label{eqn5}
f(n) &=&  e^{x-2y+z}\frac{(x^6-x^5-1)y^{16}(z^6-z^5-1)}{x^8 (y^6-y^5+1)^2 z^8},
\end{eqnarray}
\begin{eqnarray}\label{eqn6}
g(n) &=&  e^{x-2y+z}\frac{(x^6-x^5+1)y^{16}(z^6-z^5+1)}{x^8 (y^6-y^5-1)^2 z^8}.
\end{eqnarray}
As an immediate consequence of Lemma \ref{lemma1}, we have the following theorem. 
\begin{theorem}\label{thm1}
For all $n \geq 94$,
\begin{equation}\label{thm1eqn}
	f(n) < u_{n} < g(n).
\end{equation}	
\end{theorem}
We begin with the following setup.
Define 
\begin{equation}\label{eqn7}
s_1(n)=f(n-1)+f(n+1)-g(n-1)g(n+1),
\end{equation}
and
\begin{equation}\label{eqn8}
s_2(n)=g(n-1)+g(n+1)-f(n-1)f(n+1).
\end{equation}
As a corollary of Theorem \ref{thm1}, we arrive at the following inequality for $s(n)$.
\begin{corollary}\label{thm1cor}
	For $n \geq 91$, we have
	\begin{equation}\label{thm1coreqn}
	s_1(n)<s(n)<s_2(n).
	\end{equation}	
\end{corollary}
Now we interpret the Theorem \ref{Mainthm} in terms of a polynomial expression in $s(n)$ and $u_n$ (cf. \eqref{eqn2} and \eqref{eqn3}), given as follows
\begin{theorem}\label{Mainthmreform}
	For all $n \geq 42$, we have
\begin{equation}\label{Mainthmreformeqn}
s(n)u^2_{n}-2u_n+1>0.
\end{equation}	
\end{theorem} 
To prove \eqref{Mainthmreformeqn}, first it is required to estimate upper bound of $s(n)$, given by studying $s_{2}(n)$, as follows;
\begin{theorem}\label{thm2}
For all $n \geq 3$, we have
\begin{equation}\label{thm2eqn}
s_2(n)<1.
\end{equation}	
\end{theorem}
\begin{proof}
For $n \geq 3$, rewriting \eqref{eqn4}, we have
\begin{equation}\label{thm2eqn1}
r=\sqrt{y^2-2\pi^2},\ \ x=\sqrt{y^2-\pi^2},\ \ z=\sqrt{y^2+\pi^2},\ \ w=\sqrt{y^2+2\pi^2}.
\end{equation}
Expanding $r,\ x,\ z\ \text{and}\ w$ in terms of $y$, it follows that
\begin{eqnarray}
r &=& y-\frac{\pi^2}{y}-\frac{\pi^4}{y^3}-\frac{\pi^6}{2y^5}-\frac{5\pi^8}{8y^7}-\frac{7\pi^{10}}{8y^9}-\frac{21\pi^{12}}{16y^{11}}-\frac{33\pi^{14}}{16y^{13}}+O \Bigl(\frac{1}{y^{15}}\Bigr),\nonumber\\
x &=& y-\frac{\pi^2}{2y}-\frac{\pi^4}{8y^3}-\frac{\pi^6}{16y^5}-\frac{5\pi^8}{128y^7}-\frac{7\pi^{10}}{256y^9}-\frac{21\pi^{12}}{1024y^{11}}-\frac{33\pi^{14}}{2048y^{13}}+O\Bigl(\frac{1}{y^{15}}\Bigr),\nonumber\\
z &=& y+\frac{\pi^2}{2y}-\frac{\pi^4}{8y^3}+\frac{\pi^6}{16y^5}-\frac{5\pi^8}{128y^7}+\frac{7\pi^{10}}{256y^9}-\frac{21\pi^{12}}{1024y^{11}}+\frac{33\pi^{14}}{2048y^{13}}+O\Bigl(\frac{1}{y^{15}}\Bigr),\nonumber\\
w &=& y+\frac{\pi^2}{y}-\frac{\pi^4}{y^3}+\frac{\pi^6}{2y^5}-\frac{5\pi^8}{8y^7}+\frac{7\pi^{10}}{8y^9}-\frac{21\pi^{12}}{16y^{11}}+\frac{33\pi^{14}}{16y^{13}}+O\Bigl(\frac{1}{y^{15}}\Bigr)\nonumber.
\end{eqnarray}
It can be easily verified that for all $n \geq 59$,
\begin{equation}\label{thm2eqn2}
r_1<r<r_2,
\end{equation}
\begin{equation}\label{thm2eqn3}
x_1<x<x_2,
\end{equation}
\begin{equation}\label{thm2eqn4}
z_1<z<z_2,
\end{equation}
and
\begin{equation}\label{thm2eqn5}
w_1<w<w_2,
\end{equation}
with
\begin{eqnarray}
r_1&=& y-\frac{\pi^2}{y}-\frac{\pi^4}{y^3}-\frac{\pi^6}{2y^5}-\frac{5\pi^8}{8y^7}-\frac{7\pi^{10}}{8y^9}-\frac{21\pi^{12}}{16y^{11}}-\frac{34\pi^{14}}{16y^{13}},\nonumber\\
r_2 &=& y-\frac{\pi^2}{y}-\frac{\pi^4}{y^3}-\frac{\pi^6}{2y^5}-\frac{5\pi^8}{8y^7}-\frac{7\pi^{10}}{8y^9}-\frac{21\pi^{12}}{16y^{11}}-\frac{33\pi^{14}}{16y^{13}},\nonumber\\
x_1 &=& y-\frac{\pi^2}{2y}-\frac{\pi^4}{8y^3}-\frac{\pi^6}{16y^5}-\frac{5\pi^8}{128y^7}-\frac{7\pi^{10}}{256y^9}-\frac{21\pi^{12}}{1024y^{11}}-\frac{34\pi^{14}}{2048y^{13}},\nonumber\\
x_2 &=& y-\frac{\pi^2}{2y}-\frac{\pi^4}{8y^3}-\frac{\pi^6}{16y^5}-\frac{5\pi^8}{128y^7}-\frac{7\pi^{10}}{256y^9}-\frac{21\pi^{12}}{1024y^{11}}-\frac{33\pi^{14}}{2048y^{13}},\nonumber\\
z_1&=& y+\frac{\pi^2}{2y}-\frac{\pi^4}{8y^3}+\frac{\pi^6}{16y^5}-\frac{5\pi^8}{128y^7}+\frac{7\pi^{10}}{256y^9}-\frac{21\pi^{12}}{1024y^{11}},\nonumber\\
z_2 &=& y+\frac{\pi^2}{2y}-\frac{\pi^4}{8y^3}+\frac{\pi^6}{16y^5}-\frac{5\pi^8}{128y^7}+\frac{7\pi^{10}}{256y^9}-\frac{21\pi^{12}}{1024y^{11}}+\frac{33\pi^{14}}{2048y^{13}},\nonumber\\
w_1&=& y+\frac{\pi^2}{y}-\frac{\pi^4}{y^3}+\frac{\pi^6}{2y^5}-\frac{5\pi^8}{8y^7}+\frac{7\pi^{10}}{8y^9}-\frac{21\pi^{12}}{16y^{11}},\nonumber\\
w_2&=& y+\frac{\pi^2}{y}-\frac{\pi^4}{y^3}+\frac{\pi^6}{2y^5}-\frac{5\pi^8}{8y^7}+\frac{7\pi^{10}}{8y^9}-\frac{21\pi^{12}}{16y^{11}}+\frac{33\pi^{14}}{16y^{13}}.\nonumber
\end{eqnarray}
Following the definition of $s_2(n)$ given in \eqref{eqn8}, we see that it suffices to estimate $f(n-1)$, $f(n+1)$, $g(n-1)$ and $g(n+1)$. Now, we observe that each of these four functions consists of two factors, the exponential factor and the rational function in variables $x, y$ and $z$ (cf. \eqref{eqn5} and \eqref{eqn6}). This suggests that it is enough to estimate $e^{r-2x+y}$, $e^{y-2z+w}$, $h(n-1)$, $h(n+1)$, $q(n-1)$ and $q(n+1)$ individually, where
\begin{equation}\label{thm2eqn6}
f(n) = e^{x-2y+z}h(n), \ \ g(n) = e^{x-2y+z}q(n),
\end{equation}
with
\begin{equation}\label{thm2eqn7}
h(n)=\frac{(x^6-x^5-1)y^{16}(z^6-z^5-1)}{x^8 (y^6-y^5+1)^2 z^8},
\end{equation}
and 
\begin{equation}\label{thm2eqn8}
q(n)=\frac{(x^6-x^5+1)y^{16}(z^6-z^5+1)}{x^8 (y^6-y^5-1)^2 z^8}.
\end{equation}
First, let us consider the exponential factors $e^{r-2x+y}$ and $e^{y-2z+w}$. By \eqref{thm2eqn2}-\eqref{thm2eqn5}, for all $n \geq 59$, it follows that
\begin{equation}\label{thm2eqn9}
e^{r_1-2x_2+y}<e^{r-2x+y}<e^{r_2-2x_1+y},
\end{equation}
\begin{equation}\label{thm2eqn10}
e^{y-2z_2+w_1}<e^{y-2z+w}<e^{y-2z_1+w_2}.
\end{equation}
Next, we estimate \eqref{thm2eqn9} and \eqref{thm2eqn10} by Taylor expansion of the exponential function in order to get bounds in terms of rational function in $y$. For convenience, set
\begin{equation}\label{thm2eqn11}
\Phi(t)=\sum_{i=0}^{6}\frac{t^i}{i!},
\end{equation}
and 
\begin{equation}\label{thm2eqn12}
\phi(t)=\sum_{i=0}^{7}\frac{t^i}{i!}.
\end{equation}
For $t \in \mathbb{R}_{<0}$, 
\begin{equation}\label{thm2eqn13}
\phi(t)<e^t<\Phi(t).
\end{equation}
We note that
\begin{equation*}
r_2-2x_1+y=-\frac{\pi^4\Bigl(1039\pi^{10}+651\pi^8y^2+420\pi^6y^4+280\pi^4y^6+192\pi^2y^8+128y^{10}\Bigr)}{512 y^{13}} <0
\end{equation*}
and for all $n \geq 2$,
\begin{equation*}
y-2z_1+w_2=\frac{\pi^4\Bigl(1056\pi^{10}-651\pi^8y^2+420\pi^6y^4-280\pi^4y^6+192\pi^2y^8-128y^{10}\Bigr)}{512 y^{13}} <0.
\end{equation*}
Putting \eqref{thm2eqn13} into \eqref{thm2eqn9} and \eqref{thm2eqn10}, we get for $n \geq 59$,
\begin{equation}\label{thm2eqn14}
\phi(r_1-2x_2+y)<e^{r-2x+y}<\Phi(r_2-2x_1+y),
\end{equation}
and 
\begin{equation}\label{thm2eqn15}
\phi(y-2z_2+w_1)<e^{y-2z+w}<\Phi(y-2z_1+w_2).
\end{equation}
Finally, it remains to estimate $h(n-1)$, $h(n+1)$, $q(n-1)$ and $q(n+1)$. We rewrite these four functions as
\begin{equation*}
\begin{split}
h(n-1) &= \frac{x^{16}\beta(r)\beta(y)}{r^8y^8\alpha(x)^2},\ \ \ h(n+1) = \frac{z^{16}\beta(y)\beta(w)}{w^8y^8\alpha(z)^2},\\
q(n-1) &= \frac{x^{16}\alpha(r)\alpha(y)}{r^8y^8\beta(x)^2},\ \ \ q(n+1) = \frac{z^{16}\alpha(y)\alpha(w)}{w^8y^8\beta(z)^2},
\end{split}
\end{equation*}
where 
\begin{equation}\label{thm2eqn16}
\alpha(t)=t^6-t^5+1\ \ \text{and}\ \ \beta(t)=t^6-t^5-1.
\end{equation}
Using \eqref{thm2eqn2}-\eqref{thm2eqn5}, for $n \geq 59$, we put down a list of inequalities as follows
\begin{eqnarray}\label{host}
r^6-r_2r^4+1&<\alpha(r)&<r^6-r_1r^4+1,\nonumber\\
x^6-x_2x^4+1&<\alpha(x)&<x^6-x_1x^4+1,\nonumber\\
z^6-z_2z^4+1&<\alpha(z)&<z^6-z_1z^4+1,\nonumber\\
w^6-w_2w^4+1&<\alpha(w)&<w^6-w_1w^4+1,\nonumber\\
r^6-r_2r^4-1&<\beta(r)&<r^6-r_1r^4-1,\nonumber\\
x^6-x_2x^4-1&<\beta(x)&<x^6-x_1x^4-1,\nonumber\\
z^6-z_2z^4-1&<\beta(z)&<z^6-z_1z^4-1,\nonumber\\
w^6-w_2w^4-1&<\beta(w)&<w^6-w_1w^4-1,\nonumber\\
x^{12}-2x_2x^{10}+x^{10}+2x^6-2x_2x^4+1&<\alpha(x)^2&<x^{12}-2x_1x^{10}+x^{10}+2x^6-2x_1x^4+1,\nonumber\\
z^{12}-2z_2z^{10}+z^{10}+2z^6-2z_2z^4+1&<\alpha(z)^2&<z^{12}-2z_1z^{10}+z^{10}+2z^6-2z_1z^4+1,\nonumber\\
x^{12}-2x_2x^{10}+x^{10}-2x^6+2x_1x^4+1&<\beta(x)^2&<x^{12}-2x_1x^{10}+x^{10}-2x^6+2x_2x^4+1,\nonumber\\
z^{12}-2z_2z^{10}+z^{10}-2z^6+2z_1z^4+1&<\beta(z)^2&<z^{12}-2z_1z^{10}+z^{10}-2z^6+2z_2z^4+1.\nonumber\\
\end{eqnarray}
By application of the above inequalities, it follows that
\begin{equation}\label{thm2eqn17}
h(n-1)>\frac{(r^6-r_2r^4-1)x^{16}(y^6-y^5-1)}{r^8y^8(x^{12}-2x_1x^{10}+x^{10}+2x^6-2x_1x^4+1)},
\end{equation}
 \begin{equation}\label{thm2eqn18}
 h(n+1)>\frac{(w^6-w_2w^4-1)z^{16}(y^6-y^5-1)}{w^8y^8(z^{12}-2z_1z^{10}+z^{10}+2z^6-2z_1z^4+1)},
 \end{equation}
\begin{equation}\label{thm2eqn19}
q(n-1)<\frac{(r^6-r_1r^4+1)x^{16}(y^6-y^5+1)}{r^8y^8(x^{12}-2x_2x^{10}+x^{10}-2x^6+2x_1x^4+1)},
\end{equation}
\begin{equation}\label{thm2eqn20}
q(n+1)<\frac{(w^6-w_1r^4+1)z^{16}(y^6-y^5+1)}{r^8y^8(z^{12}-2z_2z^{10}+z^{10}-2z^6+2z_1z^4+1)}.
\end{equation}
Invoking \eqref{thm2eqn14}-\eqref{thm2eqn15} and \eqref{thm2eqn17}-\eqref{thm2eqn20} into \eqref{thm2eqn6}, for $n \geq 59$, we have
\begin{eqnarray}
g(n-1)&<&R_1(y)= \frac{(r^6-r_1r^4+1)x^{16}(y^6-y^5+1)\Phi(r_2-2x_1+y)}{r^8y^8(x^{12}-2x_2x^{10}+x^{10}-2x^6+2x_1x^4+1)},\nonumber\\
g(n+1)&<&R_2(y)= \frac{(w^6-w_1r^4+1)z^{16}(y^6-y^5+1)\Phi(y-2z_1+w_2)}{w^8y^8(z^{12}-2z_2z^{10}+z^{10}-2z^6+2z_1z^4+1)},\nonumber\\
f(n-1)&>&R_3(y)=\frac{(r^6-r_2r^4-1)x^{16}(y^6-y^5-1)\phi(r_1-2x_2+y)}{r^8y^8(x^{12}-2x_1x^{10}+x^{10}+2x^6-2x_1x^4+1)},\nonumber\\
f(n+1)&>&R_4(y)= \frac{(w^6-w_2w^4-1)z^{16}(y^6-y^5-1)\phi(y-2z_2+w_1)}{w^8y^8(z^{12}-2z_1z^{10}+z^{10}+2z^6-2z_1z^4+1)}.\nonumber
\end{eqnarray}
By definition of $s_2(n)$ (cf. \eqref{eqn8}), it suffices to prove that
\begin{equation}\label{thm2eqn21}
R_1(y)+R_2(y)-R_3(y)R_4(y)-1<0
\end{equation}
We can reduce $R_1(y)+R_2(y)-R_3(y)R_4(y)-1$ into a rational function in $y$; i.e,
\begin{equation}\label{thm2eqn22}
R_1(y)+R_2(y)-R_3(y)R_4(y)-1 =\frac{N_1(y)}{D_1(y)},
\end{equation}
where $N_1(y)$ and $D_1(y)$ are polynomials in $y$ with respective degree $324$ and $330$. In order to prove \eqref{thm2eqn21}, it is equivalent to prove $N_1(y)D_1(y)<0$. We write
\begin{equation}\label{thm2eqn23}
N_1(y)D_1(y)= \sum_{i=0}^{654}a_i y^i,
\end{equation}
where 
\begin{equation*}
a_{654}= 2^{348}\cdot 3^8\cdot 5^4\cdot 7^4\cdot(2^8-\pi^8) < 0.
\end{equation*}
We observe that if a polynomial, say $P(x)=\sum_{i=0}^{m}a_ix^i\in \mathbb{R}[x]$ of degree $m$ with its leading coefficient $a_m \in \mathbb{R}_{<0}$, then $P(x)$ is a decreasing function in $x$ and consequently, $P(x)<0$ for all $x\geq x_0$ where $x_0 \in \mathbb{R}$. So the only undetermined factor left over is the explicit value of $y_{0}$ such that $N_1(y)D_1(y)<0$ for all $y \geq y_0$, checked by Mathematica that $y_0=6$. We conclude the proof by numerical verification that $s_2(n)<1$ holds for $3\leq n\leq 59$.  
\end{proof}
Next, using the bound of $s_2(n)$ given in Theorem \ref{thm2}, we propose an upper bound for $g(n)$ in terms of a function of $s(n)$ that enables us to get into the proof of Theorem \ref{Mainthmreform}.
\begin{theorem}\label{thm3}
For $t \in (0,1)$, define 
	\begin{equation}\label{thm3eqn0}
	\varphi(t)=\frac{1-\sqrt{1-t}}{t}.
	\end{equation}
Then for all $n \geq 30$, we have
\begin{equation}\label{thm3eqn00}
g(n)<\varphi(s(n)).
\end{equation}
\end{theorem}
\begin{proof}
We observe that for $t \in (0,1)$, $\varphi(t)$ is an increasing function in $t$. From Corollary \ref{thm1cor} and Theorem \ref{thm2}, it suggests that we need to prove for $n \geq 91$,
\begin{equation}\label{thm3eqn1}
g(n)<\varphi(s_1(n)),
\end{equation}
or equivalently,
\begin{equation}\label{thm3eqn2}
s_1(n)g(n)^2-2g(n)+1>0.
\end{equation}
Recalling the definition of $\alpha(t)$ and $\beta(t)$ (cf. \eqref{thm2eqn16}), $s_1(n)g(n)^2-2g(n)+1$ can be written in the following form
\begin{equation}\label{thm3eqn3}
s_1(n) g(n)^2-2g(n)+1= \frac{-g_1e^{r+w-2y}+g_2e^{w+2x-3y}+g_3-2g_4e^{x-2y+z}+g_5e^{r-3y+2z}}{r^8w^8x^{16}z^{16}(x^6-x^5-1)^2(y^6-y^5-1)^4(z^6-z^5-1)^2},
\end{equation}
with 
\begin{eqnarray}
g_1&=&x^{16}y^{16}z^{16}\alpha(r)\alpha(w)\alpha(x)^2\alpha(y)^2\alpha(z)^2,\\
g_2&=&r^{8}y^{24}z^{16}\beta(w)\beta(x)^2\alpha(x)^2\beta(y)\beta(z)^2,\\
g_3&=&r^{8}w^8x^{16}z^{16}\beta(x)^2\beta(y)^4\beta(z)^2,\\
g_4&=&r^{8}w^{8}x^{8}y^{16}z^{8}\alpha(x)\beta(x)^2\beta(y)^2\alpha(z)\beta(z)^2,\\
g_5&=&w^{8}x^{16}y^{24}\beta(r)\beta(x)^2\beta(y)\beta(z)^2\alpha(z)^2.
\end{eqnarray}
Since the denominator of \eqref{thm3eqn3} is a perfect square and hence positive, therefore it is required to prove that
\begin{equation}\label{thm3eqn4}
G(y) := -g_1e^{r+w-2y}+g_2e^{w+2x-3y}+g_3-2g_4e^{x-2y+z}+g_5e^{r-3y+2z} >0.
\end{equation}
Following a similar method as used in Theorem \ref{thm2}, we first estimate the exponential terms in \eqref{thm3eqn4}. It is straightforward to observe that
\begin{eqnarray}
r_2+w_2-2y&=&-\frac{\pi^4(21\pi^8+10\pi^4y^4+8y^8)}{8y^{11}}<0,\nonumber\\
w_1+2x_1-3y&=&-\frac{\pi^4(17\pi^{10}+693\pi^8y^2-420\pi^6y^4+360\pi^4y^6-192\pi^2y^8+384y^{10})}{512y^{13}}<0\ \ \text{for}\ \ n\geq 1,\nonumber\\
x_2-2y+z_2&=& -\frac{\pi^4(21\pi^8+40\pi^4y^4+128y^8)}{512y^{11}}<0,\nonumber\\
r_1+2z_1-3y&=&-\frac{\pi^4(1088\pi^{10}+693\pi^8y^2+420\pi^6y^4+360\pi^4y^6+192\pi^2y^8+384y^{10})}{512y^{13}}<0.\nonumber
\end{eqnarray}
As a consequence, by \eqref{thm2eqn2}-\eqref{thm2eqn5} and the monotonicity property of the exponential function, for $n \geq 59$ we have
\begin{equation}\label{thm3eqn5}
e^{r+w-2y}<e^{r_2+w_2-2y}<\Phi(r_2+w_2-2y),
\end{equation}
\begin{equation}\label{thm3eqn6}
e^{x+z-2y}<e^{x_2+z_2-2y}<\Phi(x_2+z_2-2y),
\end{equation}
\begin{equation}\label{thm3eqn7}
e^{w+2x-3y}>e^{w_1+2x_1-3y}>\phi(w_1+2x_1-3y),
\end{equation}
\begin{equation}\label{thm3eqn8}
e^{r+2z-3y}>e^{r_1+2z_1-3y}>\phi(r_1+2z_1-3y).
\end{equation}
Substituting \eqref{thm3eqn5}-\eqref{thm3eqn8} into \eqref{thm3eqn4} implies that for $n \geq 59$,
\begin{equation}\label{thm3eqn9}
G(y)>-g_1\Phi(r_2+w_2-2y)+g_2\phi(w_1+2x_1-3y)+g_3-2g_4\Phi(x_2+z_2-2y)+g_5\phi(r_1+2z_1-3y).
\end{equation}
The right hand side of the above equation can be simplified further by obtaining its lower bound with the aid of employing \eqref{host} and \eqref{thm2eqn1} into the definition of $\{g_{\ell}\}_{1\leq \ell \leq5}$. More precisely, we have that for $n \geq 59$,
\begin{equation}\label{thm3eqn10}
G(y)>\frac{N_2(y)}{D_2(y)}:=\frac{\sum_{i=0}^{223}b_iy^i}{2^{101}\cdot 3^2\cdot 5 \cdot 7\ y^{119}},
\end{equation}
where $b_{223}=2^{101}\cdot 3^2\cdot 5 \cdot 7$.\\
Due to similar remark as before; i.e, if a polynomial, say $P(x)=\sum_{i=0}^{m}a_ix^i\in \mathbb{R}[x]$ of degree $m$ with its leading coefficient $a_m \in \mathbb{R}_{>0}$, then $P(x)$ is an increasing function in $x$ and consequently, $P(x)>0$ for all $x\geq x_0$ where $x_0 \in \mathbb{R}$. As an immediate consequence, we note that $G(y)>0$ by verifying that $N_2(y) >0$ for all $y \geq 5$ or equivalently for $n \geq 3$. It remains to prove \eqref{thm3eqn00} for $30\leq n\leq 91$ which is done by numerical checking in Mathematica.
\end{proof}

\section{Proof of Theorem \ref{Mainthmreform}}\label{mainthm}
\emph{Proof of Theorem \ref{Mainthmreform}}. By corollary \ref{thm1cor} and Theorem \ref{thm2}, we have $s(n)<1$ for $n \geq 91$. Define 
\begin{equation}\label{Mainthmformeqn1}
Q(t)=s(n)t^2-2t+1.
\end{equation}
To establish \eqref{Mainthmreformeqn}, we prove that for $n \geq 91$,
\begin{equation}\label{Mainthmformeqn2}
Q(u_n)>0.
\end{equation}
The quadratic equation $Q(t)=0$ has two solutions, namely
\begin{equation*}
t_0=\frac{1-\sqrt{1-s(n)}}{s(n)},\ \ \text{and}\ \ t_1=\frac{1+\sqrt{1-s(n)}}{s(n)}.
\end{equation*}
Thus $Q(t)>0$ when $t<t_0$ or $t>t_1$. From Theorem \ref{thm1} and \ref{thm3}, we have that for $n \geq 91$,
\begin{equation}\label{Mainthmformeqn3}
u_n<\varphi(s(n))
\end{equation}
Set $t_0=\varphi(s(n))$ and conclude that \eqref{Mainthmformeqn2} holds for $n \geq 91$. To confirm \eqref{Mainthmreform} for $42 \leq n\leq 91$, we can directly verify by checking numerically in Mathematica. 
\section{Conclusion}\label{Conclude}
We conclude the paper by undertaking a brief study on totally positive matrices with entries from sequences of overpartitions. Due to Engel \cite{Engel}, we know that for $n \geq 2$,
\begin{equation*}
\text{det}\ M_2(\overline{p}(n)):= \text{det}\ \begin{pmatrix}
\overline{p}(n) & \overline{p}(n+1)\\
 \overline{p}(n-1) & \overline{p}(n)\\
\end{pmatrix}>0.
\end{equation*}
 Theorem \ref{Mainthm} states that det $M_3(\overline{p}(n))>0$ for $n \geq 42$, more specifically it is worthwhile to observe that for $n\geq 3$, the determinant of each $2\times 2$ minor of the matrix $M_3(\overline{p}(n))$ is nonnegative. This construction leads to the question whether one can always construct a matrix $M_{k}(\overline{p}(n))$ of order $k$  with positive determinant if we already know its all minors (of lower order) are totally positive. More precisely,
\begin{problem}\label{prob}
For a given $k \in \mathbb{Z}_{\geq 4}$, does there exists a $n(k) \in \mathbb{Z}_{\geq 1}$ such that for $n > n(k)$,
\begin{equation}\label{probeqn}
\det\ (\overline{p}(n-i+j))_{1\leq i,j\leq k}> 0,
\end{equation}
and if \eqref{probeqn} holds true, then what is the asymptotic growth of $n(k)?$ 
\end{problem}

\begin{center}
	\textbf{Acknowledgements}
\end{center}
The author would like to express sincere gratitude to her advisor Prof. Manuel Kauers for his valuable suggestions on the paper. The research was funded by the Austrian Science Fund (FWF): W1214-N15, project DK13.

\end{document}